\pdfoutput=1
\documentclass[12pt]{article}
\usepackage{a4wide,epsfig,amsmath,amssymb,latexsym,float}
\usepackage[font=footnotesize]{caption}

\def\bftau{{\boldsymbol \tau}}

\newtheorem{theorem}{Theorem}
\newtheorem{lemma}{Lemma}

\newenvironment{proof}{\begin{trivlist}\item[]{\emph{Proof.}}}
               {\hfill$\Box$\end{trivlist}}

\begin{document}

\title{Optimal spline spaces for $L^2$ $n$-width problems with boundary conditions}
\author{Michael S. Floater\footnote{
Department of Mathematics,
University of Oslo, Moltke Moes vei 35, 0851 Oslo, Norway,
{\it email: michaelf@math.uio.no}}
\and
Espen Sande\footnote{
Same address,
{\it email: espsand@math.uio.no}.
}
}
\maketitle

\abstract{In this paper we show that, with respect to the $L^2$ norm, three classes of functions in $H^r(0,1)$, defined by certain boundary conditions, admit optimal spline spaces of all degrees $\geq r-1$, and all these spline spaces have uniform knots.}

\smallskip

\noindent {\em Math Subject Classification: }
Primary: 41A15, 47G10, 
Secondary: 41A44  

\smallskip


\section{Introduction}

Recently there has been renewed interest in using splines
of maximal smoothness, i.e. of smoothness $C^{d-1}$ for splines of degree $d$, as finite elements for solving PDEs.
This is one of the main ideas behind isogeometric analysis \cite{Buffa:14,Hughes:05,Evans:2009,Takacs:2016}.
This raises the issue of how good these splines
are at approximating functions of a certain smoothness class, especially with respect to approximation in the $L^2$ norm.
This was answered to some extent by Melkman and Micchelli \cite{Melkman:78} who studied the $L^2$ approximation of functions $u$ in the Sobolev space 
$$H^r= H^r(0,1)=\{u\in L^2(0,1) : u^{(\alpha)}\in L^2(0,1), \quad \alpha=1,2,\ldots,r \},$$
and measured the error relative to the $L^2$ norm of $u^{(r)}$. They showed that from this point of view there are two spaces of splines that are \emph{optimal},
one of degree $r-1$, the other of degree $2r-1$.
Later it was shown in \cite{Floater:17} that these two
spaces are just the first two of a whole sequence
of optimal spline spaces of degrees $lr-1$, $l=1,2,3,\ldots$. In the case $r=1$ there is therefore an optimal spline space of every degree, but whether this is true for $r\geq 2$ is an open question.

In this paper we study the related problem of approximating functions in $H^r$ subject to certain boundary conditions. Specifically, we look at
\begin{equation*}
 \begin{aligned}
  H^r_0&=\{u\in H^r:\quad u^{(k)}(0)=u^{(k)}(1)=0,\quad 0\leq k<r,\quad k\text{ even}\},\\
  H^r_1&=\{u\in H^r:\quad u^{(k)}(0)=u^{(k)}(1)=0,\quad 0\leq k<r,\quad k\text{ odd} \},\\
  H^r_2&=\{u\in H^r:\quad u^{(k)}(0)=u^{(l)}(1)=0,\quad 0\leq k, l<r,\quad k\text{ even},\quad l\text{ odd} \}.
 \end{aligned}
\end{equation*}
Our main result is to show that for all $r\geq 1$, the spaces $H^r_i$, $i=0,1,2$, admit optimal spline spaces of \emph{all} degrees $\geq r-1$. This is very similar to the numerical results reported by Evans et al. \cite{Evans:2009} regarding the degrees of the spline spaces, however their paper considered other boundary conditions (periodic conditions or no conditions). 

The derivations in \cite{Melkman:78} and \cite{Floater:17} were based on
the use of an integral operator $K$ that represents integration
$r$ times.
Roughly speaking, 
and ignoring what happens at the boundary of the interval, 
if $X_n$ is an optimal space of splines of some degree
$d$, then the space $K(X_n)$, i.e., the space generated
by integrating the splines in $X_n$, $r$ times,
is also an optimal space, consisting of splines of
degree $d+r$.

In contrast, in this paper we work only with an integral operator $K$
that represents a \emph{single} integration.
We generate optimal spline spaces for $H_i^r$, $i=0,1,2$,
by applying $K$, i.e., \emph{one} integration,
both to the initial Sobolev space $H_i^1$ and its optimal spline space, $X_n$, of degree $0$. 
This approach works for $H_i^r$, $i=0,1,2$, because, unlike $H^r$ itself,
when we apply (the right) $K$ to the functions in $H_i^r$
we get back a similar space, with $r$ increased by one.

The optimal spline spaces we obtain have the same type of boundary
conditions (odd or even derivatives are zero at the
ends of the interval) as the spaces $H^r_i$ themselves.
The splines also have uniform knots,
thus making them convenient to use in practice. In particular, some of the spline spaces corresponding to $H^r_1$ are precisely the `reduced spline spaces'
studied recently by Takacs and Takacs \cite[Section 5]{Takacs:2016} (see also the end of Section 3 in this paper).
They proved approximation estimates and inverse inequalities
for these spaces, with a view to constructing fast iterative methods for solving PDEs
in the framework of isogeometric analysis.

\section{Kolmogorov $n$-widths}
We start by formulating the concept of optimality in terms of Kolmogorov $n$-widths~\cite{Pinkus:85}.
Denote the norm and inner product on $L^2=L^2(0,1)$ by
$$ \| f\|^2 = (f,f), \qquad (f,g) = \int_0^1 f(t) g(t) \, dt, $$
for real-valued functions $f$ and $g$. 
For a subset $A$ of $L^2$, and an $n$-dimensional subspace $X_n$ of $L^2$, let
$$ E(A, X_n) = \sup_{u \in A} \inf_{v \in X_n} \|u-v\| $$
be the distance to $A$ from $X_n$ relative to the $L^2$ norm.
Then the Kolmogorov $L^2$ $n$-width
of $A$ is defined by
$$ d_n(A) = \inf_{X_n} E(A, X_n). $$
A subspace $X_n$ is called an optimal space for $A$
provided that
\begin{equation*}
d_n(A) = E(A, X_n).
\end{equation*}
Now, consider the function classes
\begin{equation}\label{eq:allA}
A^r_i=\{u\in H^r_i : \|u^{(r)}\|\leq 1\}, \quad i=0,1,2.
\end{equation}
By looking at $u/\|u^{(r)}\|$, for functions $u\in H^r_i$, we have for any $n$-dimensional subspace $X_n$ of $L^2$,
$$\|u-P_nu\|\leq E(A^r_i, X_n)\|u^{(r)}\|,$$
where $P_n$ denotes the $L^2$ projection onto $X_n$. Moreover, if $X_n$ is an optimal subspace for $A^r_i$, then
\begin{equation*}
\|u-P_nu\|\leq d_n(A^r_i)\|u^{(r)}\|,
\end{equation*}
and $d_n(A^r_i)$ is the least possible constant over all $n$-dimensional subspaces $X_n$.

\section{Main results}
We first describe the $n$-widths for $A^r_i$ in \eqref{eq:allA} and the optimal subspaces based on eigenfunctions.
We will show
\begin{theorem}\label{thm:eig}
For any integer $r\geq 1$, the $n$-widths of $A^r_i$, $i=0,1,2,$ are
\begin{equation}\label{eq:dnAi}
 d_n(A^r_0) = \frac{1}{(n+1)^r\pi^r}, \qquad d_n(A^r_1) = \frac{1}{(n\pi)^r}, \qquad  d_n(A^r_{2}) = \frac{1}{(n+\frac{1}{2})^r\pi^r}.
\end{equation}
Furthermore, the spaces
\begin{align}
 &[\sin \pi x, \sin 2\pi x, \ldots, \sin n \pi x],\label{eigspace0}\\
 &[1,\cos \pi x, \cos 2\pi x, \ldots, \cos (n-1) \pi x],\label{eigspace1}\\
 &[\sin (1/2) \pi x, \sin (3/2) \pi x, \ldots, \sin (n-1/2) \pi x]\label{eigspace2}
\end{align}
are optimal $n$-dimensional spaces for, respectively, $A^r_0$, $A^r_1$ and $A^r_2$.
\end{theorem}
Here, $[\cdots]$ denotes the span of a set of functions. The result for $A^1_1$ was shown by Kolmogorov \cite{Kolmogorov:36}. With $r$ an even number the result for $A^r_0$ was shown in \cite{Floater:17}. The remaining cases will be shown in Sections 7 and 8.

Now, let us describe the optimal spline spaces for these sets.
Suppose $\bftau = (\tau_1,\ldots,\tau_m)$ is a knot vector
such that
$$ 0 < \tau_1 < \cdots < \tau_m < 1,$$
and let $I_0 = [0,\tau_1)$,
$I_j = [\tau_j,\tau_{j+1})$,
$j=1,\ldots,m-1$, and $I_m = [\tau_m,1]$.
For any $d \ge 0$, let $\Pi_d$ be the space of polynomials of
degree at most $d$. Then we define the spline space $S_{d,\bftau}$ by
$$ S_{d,\bftau} = \{s \in C^{d-1}[0,1] : s|_{I_j} \in \Pi_d,\, j=0,1,\ldots,m \}, $$
which has dimension $m+d+1$. 
We now define the three $n$-dimensional spline spaces $S_{d,i}$, for $i=0,1,2$, by 
\begin{equation}\label{eq:allS}
\begin{aligned}
 S_{d,0} &= \{s\in S_{d,\bftau_0} : s^{(k)}(0)=s^{(k)}(1)=0,\quad 0\leq k\leq d, \quad k \text{ even}\},\\
 S_{d,1} &= \{s\in S_{d,\bftau_1} : s^{(k)}(0)=s^{(k)}(1)=0,\quad 0\leq k\leq d, \quad k \text{ odd}\},\\
 S_{d,2}&= \{s\in S_{d,\bftau_2} : s^{(k)}(0)=s^{(l)}(1)=0,\quad 0\leq k,l\leq d, \quad k \text{ even}, \quad l \text{ odd}\},
\end{aligned}
\end{equation}
where the knot vectors $\bftau_i$ for $i=0,1,2$, are given as
\begin{equation}\label{eq:alltau}
\begin{aligned}
\bftau_0 &= \begin{cases}
			(\frac{1}{n+1},\frac{2}{n+1},\ldots,\frac{n}{n+1}),\qquad &d \text{ odd},\\
			(\frac{1/2}{n+1},\frac{3/2}{n+1},\ldots,\frac{n+1/2}{n+1}),\quad &d \text{ even},
			\end{cases}\\
\bftau_1 &= \begin{cases}
			(\frac{1/2}{n},\frac{3/2}{n},\ldots,\frac{n-1/2}{n}),\qquad &d \text{ odd},\\
			(\frac{1}{n},\frac{2}{n},\ldots,\frac{n-1}{n}),\qquad &d \text{ even},
			\end{cases}\\
\bftau_2 &= \begin{cases}
			(\frac{1}{2n+1},\frac{3}{2n+1},\ldots,\frac{2n-1}{2n+1}),\quad &d \text{ even},\\
			(\frac{2}{2n+1},\frac{4}{2n+1},\ldots,\frac{2n}{2n+1}),\quad &d \text{ odd}.
			\end{cases}
\end{aligned}
\end{equation}
All these knot vectors have equidistant knots, but if we extend them to include the endpoints of $[0,1]$, the first and last knot intervals of these extended knot vectors sometimes have half the length of the interior ones. Examples of these knot vectors are shown in Figures \ref{fig:Sd1}, \ref{fig:Sd0} and \ref{fig:Sd2}.
Our main result is then the following.
\begin{theorem}\label{thm:A}
 Suppose $r\geq 1$. Then for any $i=0,1,2$, the spline spaces $S_{d,i}$ are optimal $n$-dimensional spaces for the set $A^r_i$ for any $d\geq r-1$.
\end{theorem}
The case $A^1_1$ was shown in \cite[Theorem 2]{Floater:17}. On the other hand, the case $A^r_0$ is a generalization of \cite[Theorem 1]{Floater:17} since that theorem only treated even $r$ and spline spaces of degrees $lr-1$ for $l=1,2,\ldots$, thus leaving gaps between the degrees. 
When the degree $d$ is even, the spaces $S_{d,1}$, whose common extended knot vector is equidistant, are the `reduced spline spaces' of Takacs and Takacs \cite[Section 5]{Takacs:2016}. They have also derived approximation results regarding these spaces, using Fourier analysis. We can see from Theorem \ref{thm:A} and \eqref{eq:dnAi} that, for even $d$, the constant $\sqrt{2}$ in \cite[Corollary 5.1]{Takacs:2016} can be replaced by the optimal constant $1/\pi$.

\section{Sets defined by kernels}
We need some properties of kernels, and so this section is similar to \cite[Section 3]{Floater:17}. The starting point of the analysis is to represent the lowest order function classes~$A^1_i$, $i=0,1,2,$ in the form
\begin{equation}\label{eq:Hkernel}
 A = K(B) = \{ K f : \|f\| \le 1 \},
\end{equation}
where $B$ is the unit ball in $L^2$, and $K$ is the integral operator 
$$ K f(x) = \int_0^1 K(x,y) f(y) \, dy. $$
As in \cite{Melkman:78} we use the notation $K(x,y)$ for the kernel of~$K$. We only consider kernels $K(x,y)$ that are continuous or piecewise continuous for $x, y \in [0,1]$. 
Observe that for $A$ in \eqref{eq:Hkernel} and any $n$-dimensional subspace $X_n$ of $L^2$,
\begin{equation}\label{eq:E1}
 E(A,X_n) = \sup_{\|f\| \le 1} \| (I-P_n) K f \| = \|(I-P_n)K\|_2,
\end{equation}
where $P_n$ is the orthogonal projection onto $X_n$, and $\|\cdot\|_2$ denotes the operator norm induced by the $L^2$ norm for functions. 

We will denote by $K^*$ the adjoint, or dual, of the operator $K$,
defined by
$$ (f,K^\ast g) = (Kf, g). $$
The kernel of $K^\ast$ is $K^\ast(x,y) = K(y,x)$. 
Similar to matrix multiplication,
the kernel of the composition of two integral operators $K$ and $L$
is
$$ (KL)(x,y) = (K(x,\cdot),L(\cdot,y)). $$
The operator $K^\ast K$, being self-adjoint and positive semi-definite,
has eigenvalues
\begin{equation}\label{eq:lambda}
\lambda_1 \ge \lambda_2 \ge \cdots \ge \lambda_n \ge \cdots \ge 0,
\end{equation}
and corresponding orthogonal eigenfunctions
\begin{equation}\label{eq:phi}
 K^\ast K \phi_n = \lambda_n \phi_n, \qquad n=1,2,\ldots.
\end{equation}
If we further define $\psi_n = K \phi_n$, then
\begin{equation}\label{eq:psi}
 K K^\ast \psi_n = \lambda_n \psi_n, \qquad n=1,2,\ldots,
\end{equation}
and the $\psi_n$ are also orthogonal.
The square roots of the $\lambda_n$ are known as
the $s$-numbers of $K$ (or $K^*$).
With these definitions we obtain \cite[p.~6 or p.~65]{Pinkus:85}:
\begin{theorem}\label{thm:pinkus}
$d_n(A) =\lambda_{n+1}^{1/2}$, and the space $[\psi_1,\ldots,\psi_n]$
is optimal for $A$.
\end{theorem}

\section{Totally Positive Kernels}

Melkman and Micchelli \cite{Melkman:78} proved that if $K$ is nondegenerate totally positive (NTP) \cite[p.~108]{Pinkus:85}, then there are in fact two other optimal subspaces for $A$. Specifically, if $K$ is NTP it follows from a theorem of Kellogg~\cite[p.~109]{Pinkus:85} that
the eigenvalues of $K^\ast K$ and $KK^\ast$ in (\ref{eq:phi}) and
(\ref{eq:psi}) are positive and
simple, $\lambda_1 > \lambda_2 > \cdots > \lambda_n > \cdots > 0$,
and the eigenfunctions $\phi_{n+1}$
and $\psi_{n+1}$
have exactly $n$ simple zeros in $(0,1)$,
$$ \phi_{n+1}(\xi_j) = \psi_{n+1}(\eta_j) = 0, \quad j=1,2,\ldots,n, $$
$$ 0 < \xi_1 < \xi_2 < \cdots < \xi_n < 1, \qquad
   0 < \eta_1 < \eta_2 < \cdots < \eta_n < 1. $$
Melkman and Micchelli \cite[Theorem 2.3]{Melkman:78} then proved that the spaces
\begin{equation}\label{eq:Xn0}
\begin{aligned}
X_n^0&=[K(\cdot,\xi_1),\ldots,K(\cdot,\xi_n)],\\
X_n^1&= [(KK^*)(\cdot,\eta_1),\ldots, (KK^*)(\cdot,\eta_n)]
\end{aligned}
\end{equation}
are optimal for $A$. Using a duality technique that we will discuss in the next section it was later shown in \cite[Theorem 5]{Floater:17} that, given these two optimal spaces, there is an optimal space $X_n^d$, for all $d=0,1,2,\ldots$, where
\begin{equation}\label{eq:Xn}
 X_n^d = \begin{cases}
  [(KK^\ast)^i K(\cdot,\xi_1),\ldots, (KK^\ast)^i K(\cdot,\xi_n)],
     & d = 2i, \cr
  [(KK^\ast)^{i+1}(\cdot,\eta_1),\ldots, (KK^\ast)^{i+1}(\cdot,\eta_n)],
     & d=2i+1.
   \end{cases}
\end{equation}

Melkman and Micchelli also constructed two optimal subspaces for the set $A$ even when $K$ is not NTP, but for $K$ satisfying some related properties. We will deal with such a situation in Section 8.

\section{Further optimality results}
In this section we describe how optimal subspaces for the set $A$ in \eqref{eq:Hkernel} can be used to find optimal subspaces for sets of the form $K^*(A)$, $KK^*(A)$, and so on. The results here will hold for any integral operator $K$.

To ease notation we define two function classes $A^r$ and $A^r_*$, for $r\geq 1$, by
\begin{equation}\label{eq:Arseq}
A^r=\begin{cases}
     (KK^*)^iK(B), & r=2i+1,\\
    (KK^*)^{i}(B),  & r = 2i,
   \end{cases}\qquad 
A^r_*=\begin{cases}
     (K^*K)^iK^*(B), & r=2i+1,\\
    (K^*K)^{i}(B),  & r = 2i.
   \end{cases} 
\end{equation}
Observe that both $A^r$ and $A^r_*$ are defined by alternately applying the operators $K$ and $K^*$, $r$ times, to the unit ball $B$, with $K$ always being the left-most operator for $A^r$, and $K^*$ always being the left-most operator for $A^r_*$. Since $A^1=A$, we will write $A_*$ when referring to $A_*^1$. As we shall see momentarily the duality between the operators $K$ and $K^*$ will play an important role for the sets $A^r$ and $A^r_*$, and especially their respective optimal subspaces. In some sense their optimal subspaces could be considered `dual' to each other.

Since eigenvalues of powers of $KK^*$ (and $K^*K$) are just powers of the $\lambda_n$ in \eqref{eq:lambda}, with the same corresponding eigenfunction, it follows that the $n$-widths of the sets $A^r$ and $A^r_*$ are given by
\begin{equation}\label{eq:dnAr}
d_n(A^r_*)=d_n(A^r) = d_n(A)^r,
\end{equation}
and the space $[\psi_1,\ldots,\psi_n]$ in Theorem \ref{thm:pinkus} is optimal for $A^r$, and the space $[\phi_1,\ldots,\phi_n]$ is optimal for $A^r_*$. As a tool for finding further optimal subspaces for $A^r$ and $A^r_*$, with $r\geq 2$, we start with the following lemma.

\begin{lemma}\label{lem:new}
For any integral operator $K$, let $L=KK^*$. If $X_n$ and $Y_n$ are any subspaces of $L^2$, then
\begin{equation*}
\begin{aligned}
E(A^r, L^i(X_n)) &\leq E(A,X_n)E(A^{r-1},L^i(X_n)), \qquad &&r=2i+1,\\
E(A^r, L^{i-1}K(Y_n)) &\leq E(A_*,Y_n)E(A^{r-1},L^{i-1}K(Y_n)), \qquad &&r=2i,
\end{aligned}
\end{equation*}
for $r\geq 2$.
\end{lemma}
\begin{proof}
First assume $r=2i+1$, for $i\geq 1$. From the definition of $A^r$ we have $A^r=L^iK(B)$ and $A^{r-1}=L^i(B)$. 

Let $P_n$ be the $L^2$ projection onto $X_n$, and let $Q_n$ be $L^2$ projection onto $L^i(X_n)$. 
Then 
$$L^iP_nKf\in L^i(X_n),$$ 
for all $f\in L^2$,
and so $$(I-Q_n)L^iP_nK=0.$$  
Thus, by using equation \eqref{eq:E1}, we find that
 \begin{align*}
 E(A^r, L^i(X_n))&=\|(I-Q_n)L^iK\|_2 = \|(I-Q_n)L^iK - (I-Q_n)L^iP_nK\|_2,\\
    &=\|(I-Q_n)L^i(I-P_n)K\|_2
    \leq \|(I-Q_n)L^i\|_2\|(I-P_n)K \|_2,\\
    &=E(A^{r-1},L^i(X_n))E(A, X_n).
\end{align*}

Next, assume $r=2i$, for $i\geq 1$. Then $A^r=L^{i}(B)$ and $A^{r-1}=L^{i-1}K(B)$. In this case, let $P_n$ be the $L^2$ projection onto $Y_n$, and let $Q_n$ be $L^2$ projection onto $L^{i-1}K(Y_n)$. 
Then, as before,
$$(I-Q_n)L^{i-1}KP_nK^*=0,$$
and the result follows by an almost identical argument as in the previous case.
\end{proof}

Now suppose that $X_n^0$ is an optimal $n$-dimensional subspace for $A$, and $Y_n^0$ is an optimal $n$-dimensional subspace for $A_*$. With these two subspaces one can generate a whole sequence of subspaces $X_n^d$ and $Y_n^d$, by 
\begin{equation}\label{eq:prop}
\begin{aligned}
X_n^{d}&=K(Y_n^{d-1}),\qquad Y_n^{d}&=K^*(X_n^{d-1}),
\end{aligned}
\end{equation}
for all $d=1,2,3,\ldots$, and it follows from \cite[Lemma 1]{Floater:17} that all the $X_n^d$ are optimal for the $n$-width of $A^1=A$, and all the $Y_n^d$ are optimal for the $n$-width of $A_*^1=A_*$. 
Note that for $d>0$, the spaces $X_n^d$ and $Y_n^d$ could in general have dimension less than $n$, but they are still optimal for the $n$-width problem. In fact, if $X_n^d$ or $Y_n^d$ have dimension $m$, $0\leq m<n$, then $d_m(A)$ must equal $d_n(A)$ by definition of the $n$-width.

Next, we consider $A^r$ and $A^r_*$ for $r\geq 2$.

\begin{lemma}\label{lem:new2}
Suppose the subspace $X_n^0$ is optimal for $A$ and $Y_n^0$ is optimal for $A_*$. Then, for $r\geq 2$, 
\begin{align}
E(A^r, X_n^d) \leq d_n(A)E(A^{r-1},X_n^d),\label{ineq:Xnd}\\
E(A^r_*, Y_n^d) \leq d_n(A)E(A_*^{r-1},Y_n^d)\label{ineq:Ynd},
\end{align}
for all $d\geq r-1$.
\end{lemma}
\begin{proof}
We start by proving inequality \eqref{ineq:Xnd}.
 Let $L=KK^*$. First, assume $r=2i+1$, for $i\geq 1$. It then follows from \eqref{eq:prop} that $X_n^d=L^i(X_n^{d-r+1})$ for $d\geq r-1$, and so the result follows from Lemma \ref{lem:new}, with $X_n=X_n^{d-r+1}$,
since $X_n^{d-r+1}$ is optimal for $A$.

Next, assume $r=2i$, for $i\geq 1$. It then follows from \eqref{eq:prop} that $X_n^d=L^iK(Y_n^{d-r+1})$ for $d\geq r-1$, and so the result follows from Lemma \ref{lem:new}, with $Y_n=Y_n^{d-r+1}$,
since $Y_n^{d-r+1}$ is optimal for $A_*$ and $d_n(A_*)=d_n(A)$.

Inequality \eqref{ineq:Ynd} then follows from the same argument if we interchange the roles of $K$ and $K^*$.
\end{proof}

Using Lemma \ref{lem:new2}, we now obtain optimality results for $A^r$ and $A^r_*$, for all $r\geq 1$.

\begin{theorem}\label{thm:new}
Suppose the subspace $X_n^0$ is optimal for $A$ and $Y_n^0$ is optimal for $A_*$. Then, for $r\geq 1$,
\begin{itemize}
 \item the subspaces $X_n^d$ in \eqref{eq:prop} are optimal for the $n$-width of $A^{r}$, and
 \item the subspaces $Y_n^d$ in \eqref{eq:prop} are optimal for the $n$-width of $A^{r}_*$,
\end{itemize}
for all $d\geq r-1$.
\end{theorem} 
\begin{proof}
The case $r=1$ follows from \cite[Lemma 1]{Floater:17}. For $r\geq 2$ the result for the $X_n^d$ follows from inequality \eqref{ineq:Xnd} in Lemma \ref{lem:new2}, equation \eqref{eq:dnAr} and induction on $r$, since $d_n(A)^r=d_n(A)d_n(A)^{r-1}$. Similarly, now using inequality \eqref{ineq:Ynd} in Lemma \ref{lem:new2}, we get the result for the $Y_n^d$ as well.
\end{proof}

\begin{figure}[H]
\centering\captionsetup{width=.7\linewidth}
\begin{tabular}{rrrrr}
$X_n^0$ & $Y_n^1$ & $X_n^2$ & $Y_n^3$ & $\cdots$ \\
\hline
$A^1$ & $A_*^1$ & $A^1$ & $A_*^1$ & $\cdots$ \\
      & $A_*^2$ & $A^2$ & $A_*^2$ & $\cdots$ \\
      &         & $A^3$ & $A_*^3$ & $\cdots$ \\
      &         &       & $A_*^4$ & $\cdots$
\end{tabular}\qquad
\begin{tabular}{rrrrr}
$Y_n^0$ & $X_n^1$ & $Y_n^2$ & $X_n^3$ & $\cdots$ \\
\hline
$A_*^1$ & $A^1$ & $A_*^1$ & $A^1$ & $\cdots$ \\
        & $A^2$ & $A_*^2$ & $A^2$ & $\cdots$ \\
        &       & $A_*^3$ & $A^3$ & $\cdots$ \\
        &       &         & $A^4$ & $\cdots$
\end{tabular}
\caption{Optimality results.}
\label{fig:optimality}
\end{figure}
We have summarized the statement of Theorem \ref{thm:new} in Figure \ref{fig:optimality}. Under the assumption of Theorem \ref{thm:new} on $X_n^0$ and $Y_n^0$, all the spaces (above the line) in the two tables are optimal for all the function classes below them.
Optimality of $X_n^0$ for $A^1$ implies optimality of $Y_n^1$ for $A^1_*$ by \cite[Lemma 1]{Floater:17}, and so on along the first row (below the line) in the left table. Then, by Lemma \ref{lem:new2}, optimality of $X_n^0$ for $A^1$, and $Y_n^1$ for $A_*^1$, imply optimality of $Y_n^1$ for $A^2_*$, and so on along the second row. Optimality of $X_n^0$ for $A^1$, and $X_n^2$ for $A^2$, imply optimality of $X_n^2$ for $A^3$, and so on along the third row. 
Similarly for the right table.

Let us now turn back to the case where $K$ is NTP. The subspace $X_n^0$ in \eqref{eq:Xn0} is optimal for $A$, and since $K$ being NTP is equivalent to $K^*$ being NTP, we also have that the subspace 
\begin{equation}\label{eq:Yn0}
Y_n^0=[K^*(\cdot,\eta_1),\ldots, K^*(\cdot,\eta_n)]
\end{equation}
is optimal for $A_*$, and so we can apply Theorem~\ref{thm:new}. The subspaces $X_n^d$ in \eqref{eq:prop} are in this case the same as those in equation \eqref{eq:Xn}. Since the eigenvalues \eqref{eq:lambda} (and thus also the $n$-widths) are strictly decreasing whenever $K$ is NTP, the subspaces $X_n^d$ and $Y_n^d$ are in this case also $n$-dimensional for all $d\geq 0$.

\section{Mixed boundary conditions}
In this section we study the $n$-width problem for the function class $A^r_2$ in \eqref{eq:allA}. Consider the operator $K$ given by
\begin{equation}\label{eq:integralK}
Kf(x)=\int_0^xf(y)d y = \int_0^1K(x,y)f(y)dy,
\end{equation}
whose kernel is
\begin{equation}\label{eq:lowestK}
 K(x,y)= \begin{cases} 0 & x < y, \cr
                         1 & x \ge y.
           \end{cases}
\end{equation}
Using the equality $K^*(x,y)=K(y,x)$, we find that
\begin{equation}\label{eq:integralKs}
K^*f(x)=\int_x^1f(y)d y.
\end{equation}
Thus $K$ represents integration from the left, while $K^*$ represents integration from the right.

From \eqref{eq:integralK} we see that the set $A^1_{2}$ in equation \eqref{eq:allA} can be expressed as
\begin{align*}
 A^1_{2}  =\{ u\in H^1 : \|u'\|\leq 1, \,\, u(0)=0\} = \{\int_0^x f(y) d y : \|f\|\leq 1\} = K(B). 
\end{align*}
To see that the remaining $A^r_2$ can be expressed in terms of $K$ and $K^*$,
it is convenient to recognize the kernel of the
composition $KK^*$ as the Green's function for a boundary value problem,
whose eigenfunctions we will need later anyway [in equation \eqref{eq:Kpsi}].
\begin{lemma}\label{lem:one}
If $u(x) = KK^* f(x)$ then
$u$ is the unique solution to the boundary value problem
\begin{align}\label{eq:bvp1}
 - u''(x) = f(x), \quad u(0) = u'(1) = 0.
\end{align}
\end{lemma}
\begin{proof}
We see from \eqref{eq:integralK} and \eqref{eq:integralKs} that
for any $h$,
\begin{align}
 (Kh)'(x) &= h(x), \label{eq:Kh} \\
 (K^* h)'(x) &= - h(x). \nonumber
\end{align}
So, if  $u(x) = KK^* f(x)$ then $- u''(x) = f(x)$.
For the left boundary condition, from (\ref{eq:integralK}), we find that
$$ u(0)=(KK^*f)(0)=0. $$
For the right boundary condition,
from (\ref{eq:Kh}) and (\ref{eq:integralKs}),
$$ u'(1)=(KK^*f)'(1)=(K^*f)(1)=0. $$
To see that $u$ is unique, suppose $f=0$ in \eqref{eq:bvp1}. 
Then $u$ must be a linear function, but to satisfy the boundary conditions we must have $u=0$.
\end{proof}

By applying the above lemma to functions $f$ in $B$ and $K(B)$
respectively and repeating the procedure $i$ times, we find that
\begin{equation*}
A^{2i}_{2} = (KK^*)^{i}(B),\qquad A^{2i+1}_{2} = (KK^*)^{i}K(B),
\end{equation*}
where $A^{2i}_{2}$ and $A^{2i+1}_{2}$ are as in equation \eqref{eq:allA}. Observe that the left-most operator for the function class $A^r_2$ is always $K$, and so $A^r_2$ is an instance of $A^r$ in \eqref{eq:Arseq}.

\subsection{Proof of Theorem \ref{thm:eig} for $A^r_2$}
In analogy to Lemma \ref{lem:one} we have,
for the other composition $K^*K$,
\begin{lemma}\label{lem:two}
If $u(x) = K^*K f(x)$ then
$u$ is the unique solution to the boundary value problem
\begin{align*}
 - u''(x) = f(x), \quad u'(0) = u(1) = 0.
\end{align*}
\end{lemma}
From Lemma \ref{lem:two}, we see that 
the eigenvalues and eigenfunctions of $K^*K$ are
\begin{equation}\label{eq:Kphi}
\begin{aligned}
\lambda_n = \frac{1}{(n-1/2)^2 \pi^2}, \qquad \phi_n(x) = \cos (n-1/2) \pi x, \qquad n=1,2,\ldots. 
\end{aligned}
\end{equation}
From Lemma \ref{lem:one}, the operator $KK^*$ has the same eigenvalues, but the
eigenfunctions are
\begin{equation}\label{eq:Kpsi}
\psi_n(x) = \sin (n-1/2) \pi x, \qquad n=1,2,\ldots.
\end{equation}
So, by Theorem \ref{thm:pinkus},
the $n$-width of $A^1_2$ is as given in equation \eqref{eq:dnAi}
and an optimal subspace is as given in \eqref{eigspace2}. The analogous results for $A^r_2$, $r>1$, follow from equation \eqref{eq:dnAr}.

\subsection{Proof of Theorem \ref{thm:A} for $A^r_2$}

We have already seen that the function class $A^r_2$ is the function class $A^r$ in \eqref{eq:Arseq} when $K$ has a kernel as given in \eqref{eq:lowestK}. Since it is well known that this choice of $K$ is NTP~\cite[p.~16]{Karlin:68}, we can apply Theorem~\ref{thm:new} to the spaces $X_n^0$ in \eqref{eq:Xn0} and $Y_n^0$ in \eqref{eq:Yn0}. All that remains to show is that the optimal subspaces $X_n^d$ generated as in equation~\eqref{eq:prop} are the spline spaces we claim.

The zeros $\xi_j$ of $\phi_{n+1}(x)$ in \eqref{eq:Kphi} are the knots in the even degree case for the knot vector $\bftau_2$ in equation \eqref{eq:alltau}, and the zeros $\eta_j$ of $\psi_{n+1}(x)$ in \eqref{eq:Kpsi} are the knots in the odd degree case.
Thus, $X_n^0$ in \eqref{eq:Xn0}, with the kernel of $K$ as in equation \eqref{eq:lowestK}, is equal to
\begin{align*}
 X_n^0 = [K(\cdot,\xi_1),\ldots, K(\cdot,\xi_n)] = S_{0,2},
\end{align*}
where $S_{0,2}$ is the piecewise constant spline space given in equation \eqref{eq:allS}. To find $X_n^1$ we perform a simple calculation to see that
\begin{align*}
 KK^*(x,y) = (K(x,\cdot), K(y,\cdot)) = \begin{cases}
                                               x, & x<y,\\
                                               y, & x>y,
                                              \end{cases}
\end{align*}
and so,
 $X_n^1 =K(Y_n^0)=[(KK^*)(\cdot,\eta_1),\ldots, (KK^*)(\cdot,\eta_n)] = S_{1,2},$
 the piecewise linear spline space given in equation \eqref{eq:allS}.
The remaining $X_n^d$, for $d\geq 2$, can be found by using the fact that $X_n^{d+2} = KK^*(X_n^d)$ and applying Lemma \ref{lem:one}, since the derivative of a spline is a spline on the same knot vector of one degree lower.

\emph{Remark:} We note that interchanging the roles of $K$ and $K^*$ shows that the subspaces $Y_n^d$ are optimal for the sets defined by interchanging the boundary conditions in $A^r_{2}$, i.e., odd derivatives set to zero at the left-hand side, and even derivatives set to zero at the right-hand side. One finds that the subspaces $Y_n^d$ are equal to their corresponding `dual' subspace $X_n^d$, just with interchanged boundary conditions and interchanged knots (i.e., replacing the even degree case for $\bftau_2$ in \eqref{eq:alltau} with the odd degree case, and vice versa).

\section{Symmetric boundary conditions}
In this section we study the $n$-width problems for the remaining function classes $A^r_0$ and $A^r_1$ in \eqref{eq:allA}.
Let $K_1$ be the operator given by
\begin{equation}\label{eq:K1}
 K_1 = (I-Q)K,
\end{equation}
where $Q$ is the orthogonal projection
onto the constant functions, $\Pi_{0}$, and $K$ is again the operator \eqref{eq:integralK}. 
From \cite{Melkman:78} we know that the set $A^1_1$ given in equation \eqref{eq:allA} can be written as the orthogonal sum
\begin{equation*}
 A^1_1= \Pi_0 \oplus K_1(B).
\end{equation*}
It follows from~\cite[Chap.~IV, Sec.~3.2]{Pinkus:85} that the set $A^1_0$ in equation \eqref{eq:allA} can be written as
 \begin{align*}
 A^1_0&=\{ u\in H^1 : \|u'\|\leq 1, \quad u(0)=u(1)=0 \},\\
 &= \{K^*f : \|f\|\leq 1,\quad f\perp 1\}
 =K^*(I-Q)(B)= K_1^*(B).
\end{align*}
The kernel of $K_1K_1^*$ is the Green's function to the boundary value problem
\begin{align}\label{eq:bvp3}
 -u''(x)=f(x),\quad u'(0)=u'(1)=0,\quad u,f\perp 1
\end{align}
(see e.g. \cite[Lemma 4]{Floater:17}). Using equation
\eqref{eq:bvp3} $i$ times and then adding back the constants we find that,
\begin{equation}\label{eq:Ar1}
A^{2i}_1 = \Pi_0\oplus (K_1K_1^*)^i(B),\qquad A^{2i+1}_1 = \Pi_0\oplus (K_1K_1^*)^iK_1(B),
\end{equation}
where $A^{2i}_1$ and $A^{2i+1}_1$ are as in equation \eqref{eq:allA}.
The kernel of $K_1^*K_1$ is the Green's function to the boundary value problem
\begin{align}\label{eq:bvp4}
 -u''(x)=f(x),\quad u(0)=u(1)=0
\end{align}
(see e.g. \cite[Lemma 3]{Floater:17}).
Then, using equation \eqref{eq:bvp4} $i$ times we find that,
\begin{equation*}
A^{2i}_0=(K_1^*K_1)^i(B), \qquad A^{2i+1}_0=(K_1^*K_1)^iK_1^*(B),
\end{equation*}
where $A^{2i}_0$ and $A^{2i+1}_0$ are as in equation \eqref{eq:allA}. Observe that the left-most operator for the function class $A^r_0$ is always $K_1^*$, and so $A^r_0$ is an instance of $A^r_*$ in \eqref{eq:Arseq}. The function class $A^r_1$, on the other hand, is not quite an instance of $A^r$ in \eqref{eq:Arseq}, but it is of the form $\Pi\oplus A^r$.

\subsection{Proof of Theorem \ref{thm:eig} for $A^r_0$ and $A^r_1$}
From equation \eqref{eq:bvp4} we see that 
the eigenvalues and eigenfunctions of $K_1^*K_1$ are
\begin{equation}\label{eq:Llambda}
\begin{aligned}
\lambda_n = \frac{1}{(n\pi)^2}, \qquad \phi_n(x) = \sin (n \pi x), \qquad n=1,2,\ldots.
\end{aligned}
\end{equation}
The operator $K_1K_1^*$ has the same eigenvalues, but the
eigenfunctions are
\begin{equation}\label{eq:Lpsi}
\psi_n(x) = \cos (n \pi x), \qquad n=1,2,\ldots.
\end{equation}
So, by Theorem \ref{thm:pinkus},
the $n$-widths of both $A^1_0=K_1^*(B)$ and $K_1(B)$ are equal to $d_n(A^1_0)$ in equation \eqref{eq:dnAi}.
An optimal $n$-dimensional subspace for $A^1_0$ is as given in \eqref{eigspace0},
and an optimal $n$-dimensional subspace for $K_1(B)$ is
$ [\cos (\pi x), \cos (2 \pi x), \ldots, \cos (n \pi x)].$
Since this subspace is orthogonal to $\Pi_0$ it follows that an optimal $(n+1)$-dimensional subspace for $A^1_1=\Pi_0\oplus K_1(B)$ is
$$ [1, \cos (\pi x), \cos (2 \pi x), \ldots, \cos (n \pi x)],$$
thus showing that \eqref{eigspace1} is an optimal $n$-dimensional space, and that the $n$-width of $A^1_1$ is as given in \eqref{eq:dnAi}. Pay special attention to this index-shift caused by $\Pi_0$: the $n$-width of $ K_1(B)$ is equal to $\lambda_{n+1}^{1/2}$, but the $n$-width of $A^1_1$ is equal to $\lambda_{n}^{1/2}$ in \eqref{eq:Llambda}.
As before, the analogous results for $A^r_0$ and $A^r_1$, $r>1$, follow from equation \eqref{eq:dnAr}.

\subsection{Proof of Theorem \ref{thm:A} for $A^r_0$ and $A^r_1$}
To prove Theorem \ref{thm:A} for $A^r_0$ and $A^r_1$ we will use Theorem \ref{thm:new} with $K_1$ playing the role of the generic operator $K$.
We must therefore identify the first optimal space $X_n^0$ for $K_1(B)$ and the first optimal space $Y_n^0$ for $A^1_0=K_1^*(B)$.
Unlike $K$ in equation \eqref{eq:integralK}, $K_1$ is not NTP (specifically, it is not totally positive) and this creates an extra challenge compared with subsection 7.2.
Fortunately, as shown in \cite{Melkman:78,Micchelli:77} the operator $K_1^*K_1$ is in fact NTP, and we can make use of this and other results in \cite[Section~5]{Melkman:78}.
Specifically, we have from \cite[Theorem 5.1]{Melkman:78} that
\begin{align*}
X_{n}^{0} =[K_1(\cdot,\xi_1),\ldots,K_1(\cdot,\xi_{n})]
\end{align*}
is an optimal subspace for the $n$-width of $K_1(B)$, where the $\xi_j$, for $j=1,2,\ldots,n$, are the $n$ zeros of $\phi_{n+1}(x)$ in \eqref{eq:Llambda}. Observe, that these $\xi_j$'s are the knots in the odd degree case for the knot vector $\bftau_0$ in \eqref{eq:alltau}.

Now, we consider $Y_n^0$. First, let $\eta_j$, for $j=1,2,\ldots,n+1$, be the $n+1$ zeros of $\psi_{n+1}(x)$ in \eqref{eq:Lpsi}, which are the knots in the even degree case of $\bftau_0$ in \eqref{eq:alltau}.
Additionally, let $J$ be the interpolation operator from
$C[0,1]$ to $\Pi_{0}$ determined by interpolating at $\eta_1$,
and define the operator $$\overline K_1 = (I-J)K,$$ where $K$ still is the operator \eqref{eq:integralK}.
If we let $Y_n^0$ be the $n$-dimensional space
 \begin{align}\label{eq:newYn1}
  Y_{n}^0 =
     [\overline K_1^*(\cdot,\eta_{2}),\ldots,\overline K_1^*(\cdot,\eta_{n+1}) ],
\end{align}
 then the proof of \cite[Theorem 5.1]{Melkman:78} (or \cite[Theorem 5.11 p.~121]{Pinkus:85}) contains the following important result.
\begin{lemma}\label{lem:Melkman}
If $P_n$ is the orthogonal projection onto the space $Y_n^0$, and $\lambda_{n+1}$ is as in \eqref{eq:Llambda}, then
  \begin{equation}\label{ineq:Pinkus}
  \sup_{\|f\|\leq 1}\|\overline K_1(I-P_n)f\|\leq \lambda_{n+1}^{1/2}.
\end{equation}
\end{lemma}
\begin{proof}
The operator $K_1$ is a special case of the operator $K_1$ in \cite{Pinkus:85} (as explained on page 124). It therefore satisfies the assumptions of \cite[Theorem 5.11 p.~121]{Pinkus:85}, and inequality \eqref{ineq:Pinkus} is then proved on page 122. Note the shift in index, the above $n+1$ corresponds to $n$ in \cite[Theorem 5.11]{Pinkus:85}.
\end{proof}
Using this inequality we can show the following.
\begin{theorem}\label{thm:Y}
The space $Y_n^0$ is optimal for $A_0^1=K_1^*(B)$.
\end{theorem}
\begin{proof}
Let $P_n$ be the orthogonal projection onto $Y_n^{0}$ in \eqref{eq:newYn1}. To prove that $Y_n^{0}$ is an optimal subspace for $A_0^1$, we need to show that
\begin{equation*}
  E(A_0^1, Y_n^{0})\leq d_n(A_0^1),
\end{equation*}
or equivalently,
\begin{equation*}
  \|(I-P_n)K_1^*\|_2\leq \lambda_{n+1}^{1/2},
\end{equation*}
with $\lambda_{n+1}$ as given in \eqref{eq:Llambda}. 
First observe that 
\begin{equation*}
 \|(I-P_n)K_1^*\|_2= \|K_1(I-P_n)\|_2= \|(I-Q)K(I-P_n)\|_2.
\end{equation*}
Next, since both $J$ and $Q$ are projections onto the constants, $\Pi_0$, but only $Q$ is the orthogonal projection, we must have
\begin{equation*}
  \|(I-Q)K(I-P_n)\|_2\leq \|(I-J)K(I-P_n)\|_2=\|\overline K_1(I-P_n)\|_2.
\end{equation*}
Hence, the result follows from Lemma \ref{lem:Melkman}.
\end{proof}

\emph{Remark:} Melkman and Micchelli \cite[Theorem 5.1]{Melkman:78} used the inequality in Lemma~\ref{lem:Melkman} to directly conclude that the $(n+1)$-dimensional space
\begin{equation*}
 \Pi_{0} + [(\overline K_1\, \overline K_1^*)(\cdot,\eta_{2}),\ldots,(\overline K_1\, \overline K_1^*)(\cdot,\eta_{n+1}) ],
\end{equation*}
is optimal for the set $A^1_1=\Pi_{0}\oplus K_1(B)$. On the other hand, from \cite[Lemma 1]{Floater:17} and the above Theorem \ref{thm:Y}, it follows that $K_1(Y_{n}^{0})$ is an optimal space for $K_1(B)$, and so
\begin{equation*}
   \Pi_{0}\oplus K_1(Y_{n}^{0})=\Pi_{0}\oplus[(K_1\overline K_1^*)(\cdot,\eta_{2}),\ldots,(K_1\overline K_1^*)(\cdot,\eta_{n+1})],
\end{equation*}
is optimal for $A^1_1$. This is consistent with their result, since the difference 
$$(K_1\overline K_1^*)(\cdot,\eta_{j})-(\overline K_1\,\overline K_1^*)(\cdot,\eta_{j}),$$
is a constant for any $j=2,\ldots,n+1$.

We now have the first optimal space $X_n^0$ for $K_1(B)$ and the first optimal space $Y_n^0$ for $A^1_0=K_1^*(B)$, and so we can apply Theorem \ref{thm:new}.
To do this let us express $A^r_1$ in equation \eqref{eq:Ar1} as
$$A^r_1=\Pi_0\oplus \tilde{A^r_1}.$$
Now, if $X_n^d$ and $Y_n^d$ are generated as in \eqref{eq:prop} with $K_1$ playing the role of the generic $K$, then it follows from Theorem \ref{thm:new} that, for all $r\geq 1$,
\begin{itemize}
 \item the $n$-dimensional spaces $X_n^d$ are optimal for the $n$-width of $\tilde{A^r_1}$, and
 \item the $n$-dimensional spaces $Y_n^d$ are optimal for the $n$-width of $A^r_0$,
\end{itemize}
for all $d\geq r-1$. Note further that these spaces are all $n$-dimensional since the $n$-widths in \eqref{eq:dnAi} are strictly decreasing.
Moreover, since both $X_n^d\perp\Pi_{0}$ and $\tilde{A^r_1}\perp\Pi_{0}$, we find that the $(n+1)$-dimensional spaces $\Pi_{0}\oplus X_n^d$ are optimal for $A^r_1=\Pi_{0}\oplus \tilde{A^r_1}$ for $d\geq r-1$.

The remaining task is to recognize the spaces $\Pi_0\oplus X_n^d$ and $Y_n^d$ as spline spaces.
As already stated, the optimal spaces $\Pi_0\oplus X_n^d$ were identified in \cite{Floater:17} and we have the equality
\begin{equation*}
 S_{d,1} = \Pi_0\oplus X_{n-1}^{d},
\end{equation*}
where $S_{d,1}$ is the $n$-dimensional space defined in \eqref{eq:allS}.
 However, only the spline spaces $Y_n^d$ when $d$ is odd were found in \cite{Floater:17}. In that case we have
\begin{equation}\label{eq:Ys}
 S_{d,0}=Y_n^{d},
\end{equation}
with $S_{d,0}$ also as in \eqref{eq:allS}. Now, using the definition of $\overline K_1$ we find that the kernel $\overline K_1^*(x,y) = \overline K_1(y,x)$ is equal to
\begin{equation*}
 \overline K_1^*(x,y) = \begin{cases}
                     0, \quad x<\eta_1,
                     \\
                     1, \quad \eta_1<x<y,
                     \\
                     0, \quad x>y,
                    \end{cases}
\end{equation*}
for $y>\eta_1$. The space $Y_n^0$ in equation \eqref{eq:newYn1} is then the space of piecewise constant splines with knots $\eta_j$, $j=1,\ldots,n+1$, that vanish on the intervals $[0,\eta_1)$ and $(\eta_{n+1},1]$. Since $Y_n^2=K_1^*K_1(Y_n^0)$, and so on, we know from \eqref{eq:bvp4} that equation \eqref{eq:Ys} also holds in the case of $d$ even.
This proves Theorem \ref{thm:A} for $A^r_0$ and $A^r_1$.

\section{Basis functions}
In this section we describe how to create a local basis for the spline spaces $S_{d,i},$ $i=0,1,2$. First consider $i=1$. An explanation of how to construct a local basis for $S_{d,1}$ (with~$d$ even) is presented in \cite{Takacs:2016}.
The basic idea consists of three parts. Start with our uniform knot vector $\bftau_1$ in \eqref{eq:alltau} and extend it to a uniform knot vector on the whole real line. Second, construct all the $B$-splines on this infinite knot vector that have non-zero support on $(0,1)$. Third, identify the $B$-splines that cross the boundary and add them together in pairs, chosen in such a manner that the symmetry of uniform $B$-splines ensures the boundary conditions (all odd derivatives set to zero) are satisfied. Figure \ref{fig:Sd1} shows the basis functions for $S_{d,1}$ of degree $0$ to $3$ with knot-distance $0.2$ ($n=5$). 

Next, we consider $i=0$. Constructing a basis for $S_{d,0}$ can be done by essentially the same procedure as for $S_{d,1}$. Instead of adding pairs of $B$-splines together we take differences. The symmetry of uniform $B$-splines will again ensure that the boundary conditions (all even derivatives set to zero) are satisfied. Figure \ref{fig:Sd0} shows the basis functions for $S_{d,0}$ of degree $0$ to $3$ with knot-distance $0.2$ ($n=4$).

\begin{figure}
  \centering\captionsetup{width=.5\linewidth}
      \includegraphics[width=.4\linewidth]{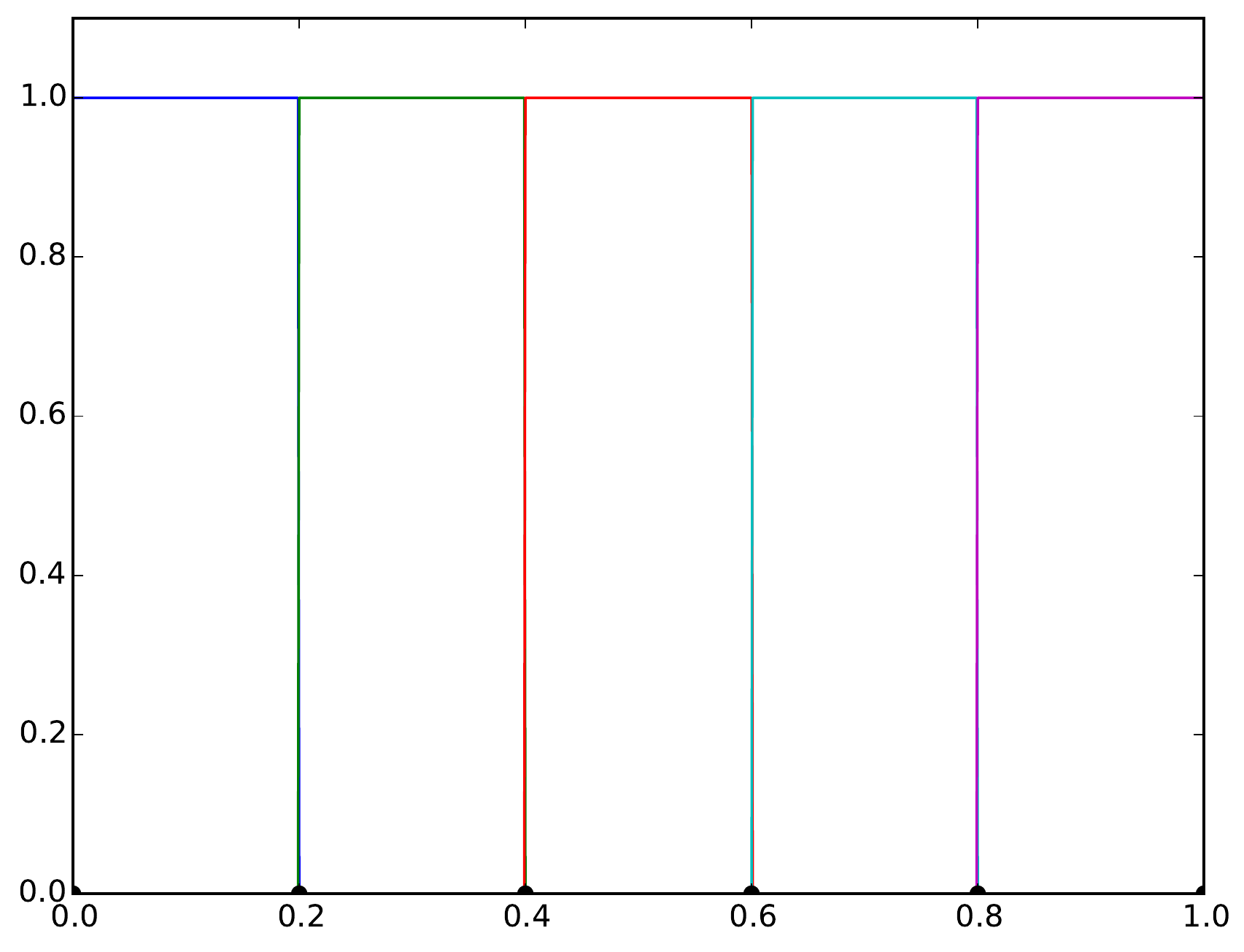}
      \includegraphics[width=.4\linewidth]{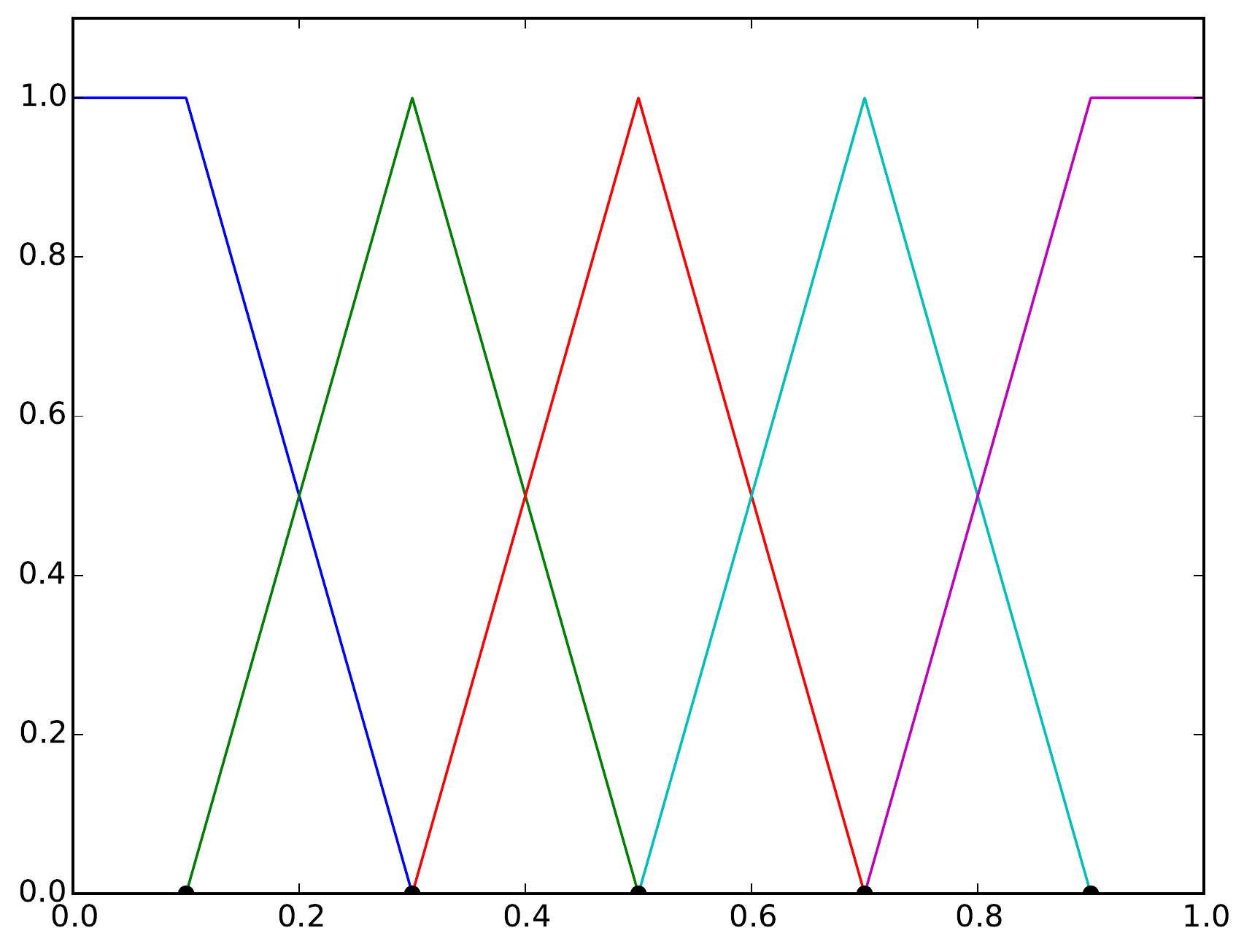}
      
      \includegraphics[width=.4\linewidth]{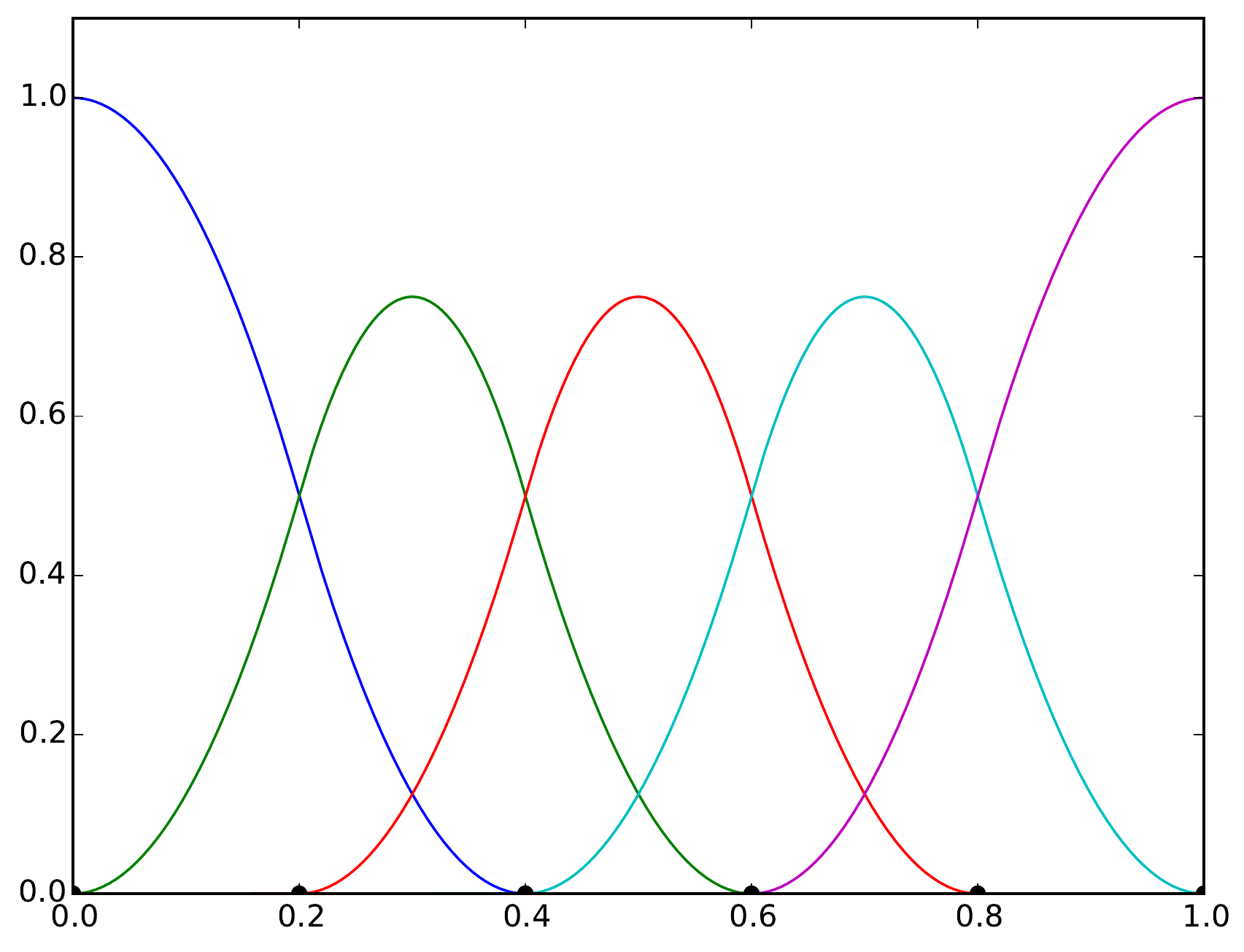}
      \includegraphics[width=.4\linewidth]{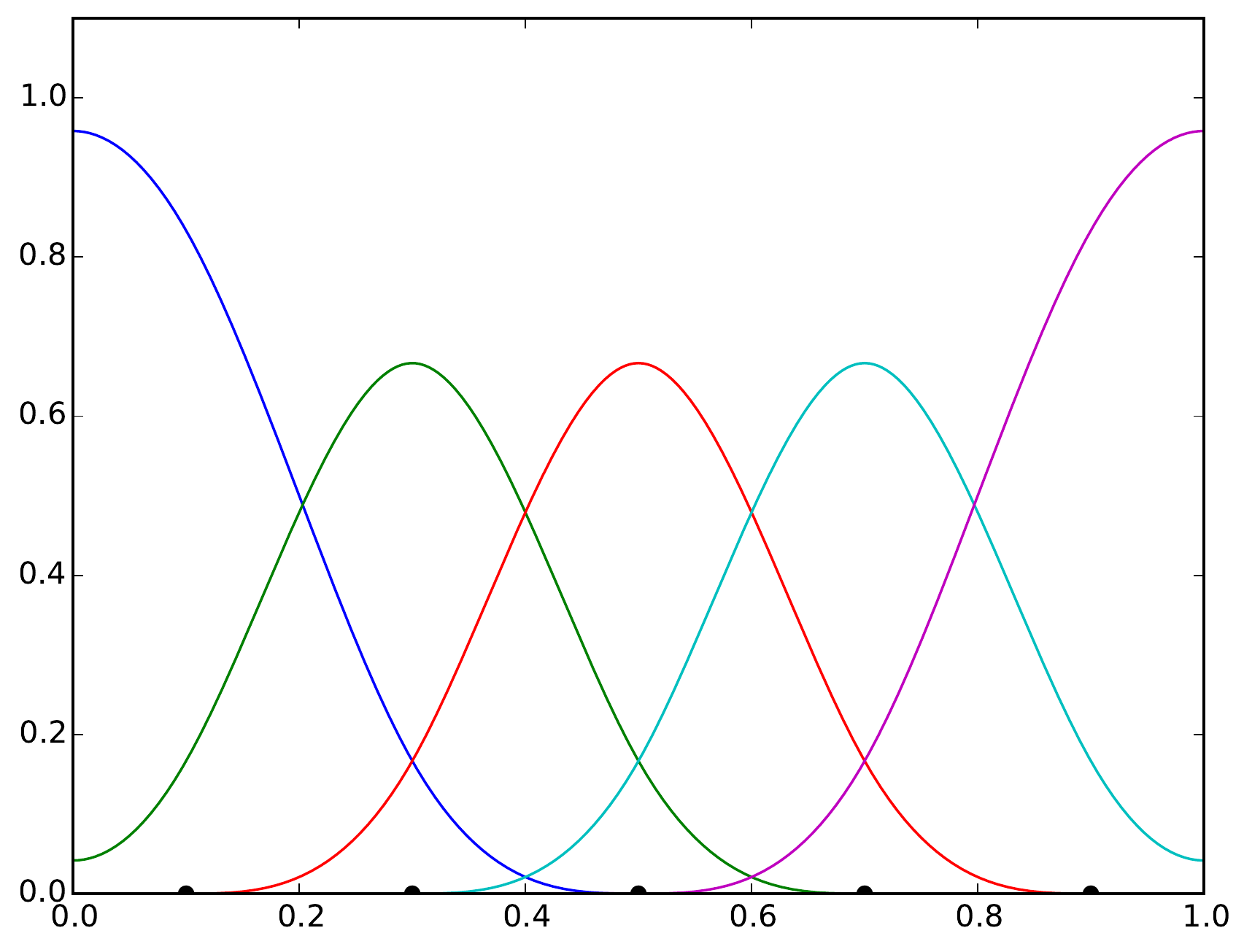}
  \caption{Basis functions for $S_{d,1}$.}
  \label{fig:Sd1}

      \includegraphics[width=.4\linewidth]{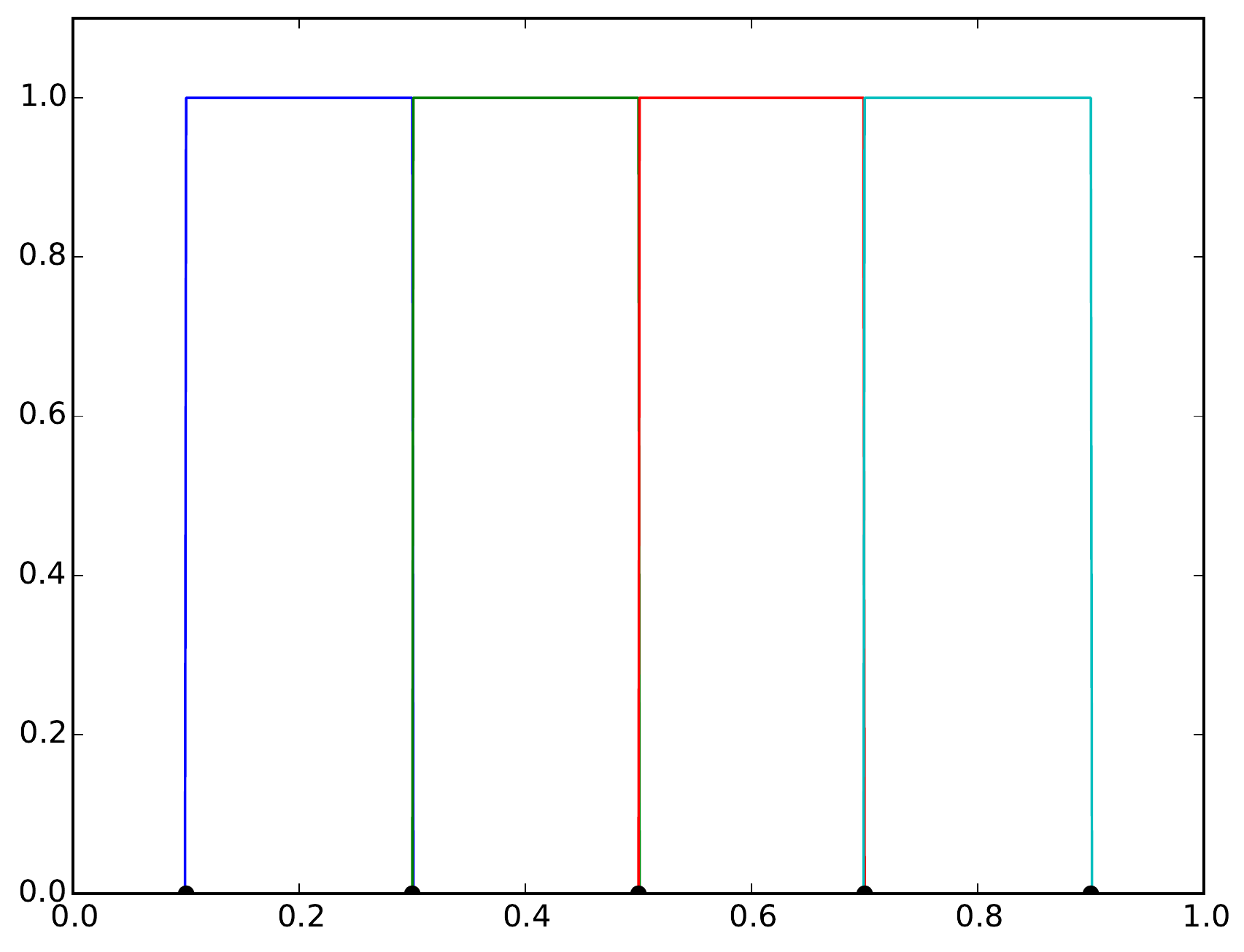}
      \includegraphics[width=.4\linewidth]{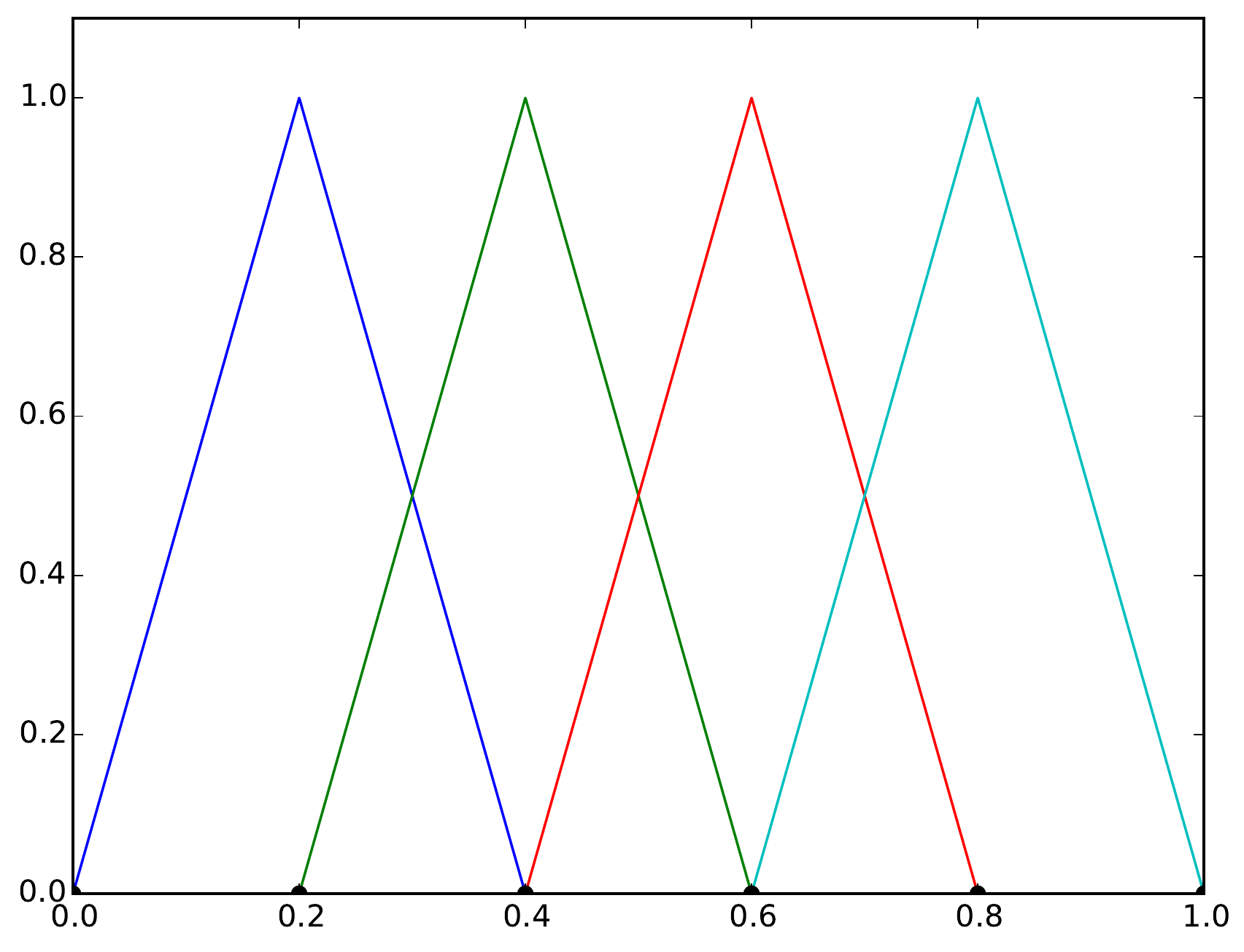}
      
      \includegraphics[width=.4\linewidth]{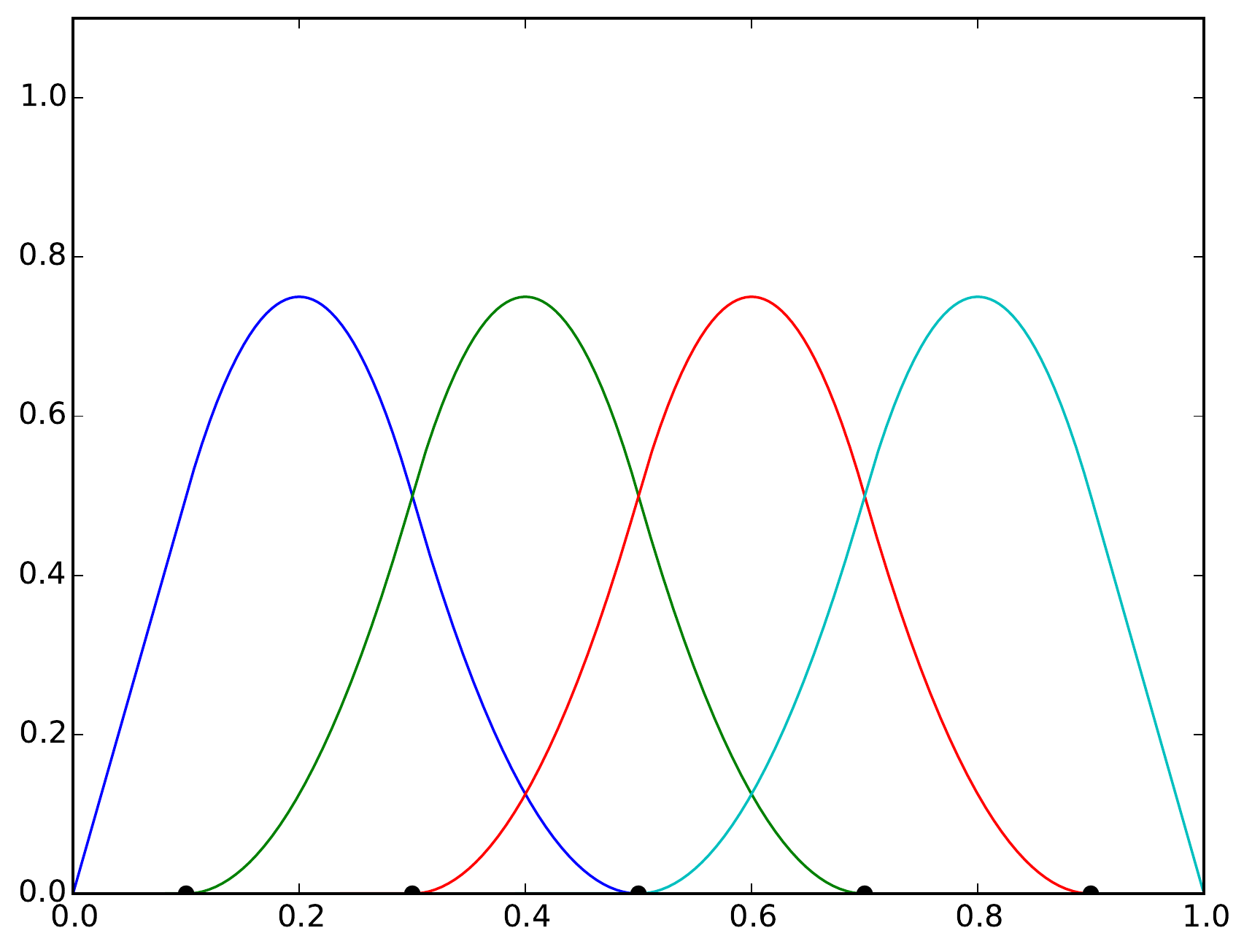}
      \includegraphics[width=.4\linewidth]{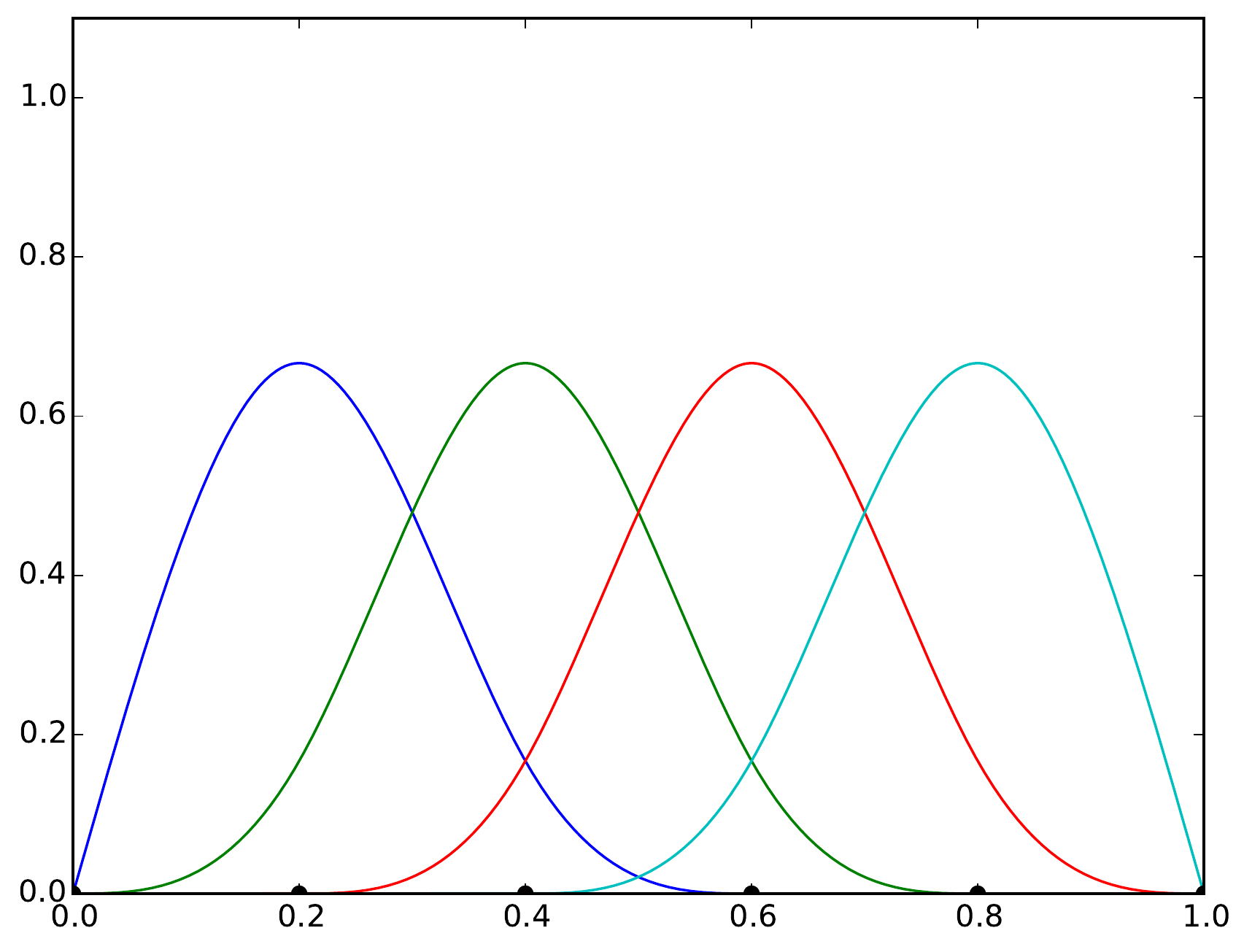}
  \caption{Basis functions for $S_{d,0}$.}
  \label{fig:Sd0}
\end{figure}

Regarding $i=2$, adding pairs of $B$-splines together on the right-hand side and subtracting on the left-hand side will give a basis for $S_{d,2}$. Figure \ref{fig:Sd2} shows the basis functions for $S_{d,2}$ of degree $0$ to $3$ with knot-distance $2/9$ ($n=4$).

\begin{figure}
  \centering\captionsetup{width=.5\linewidth}
  	  \includegraphics[width=.4\linewidth]{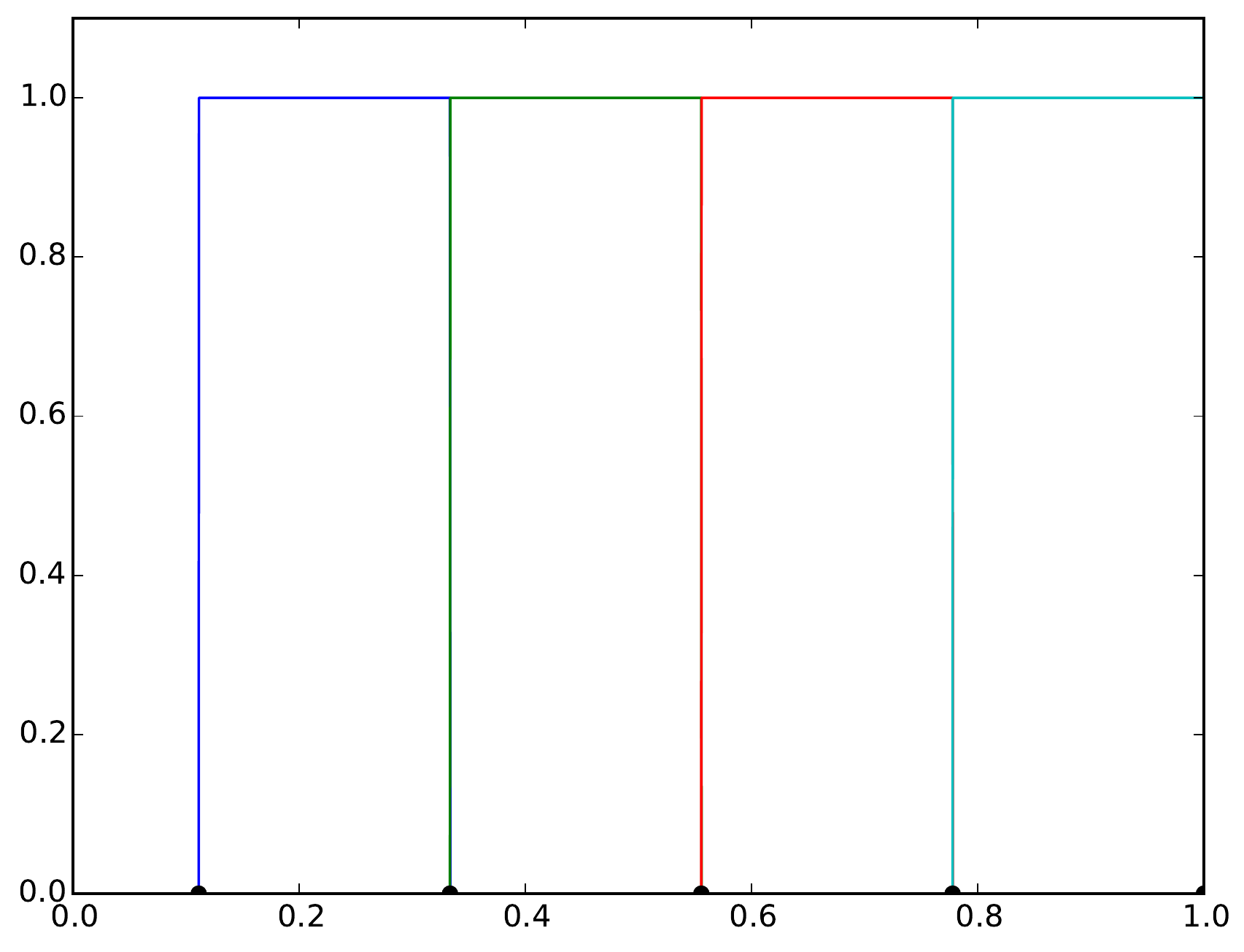}
      \includegraphics[width=.4\linewidth]{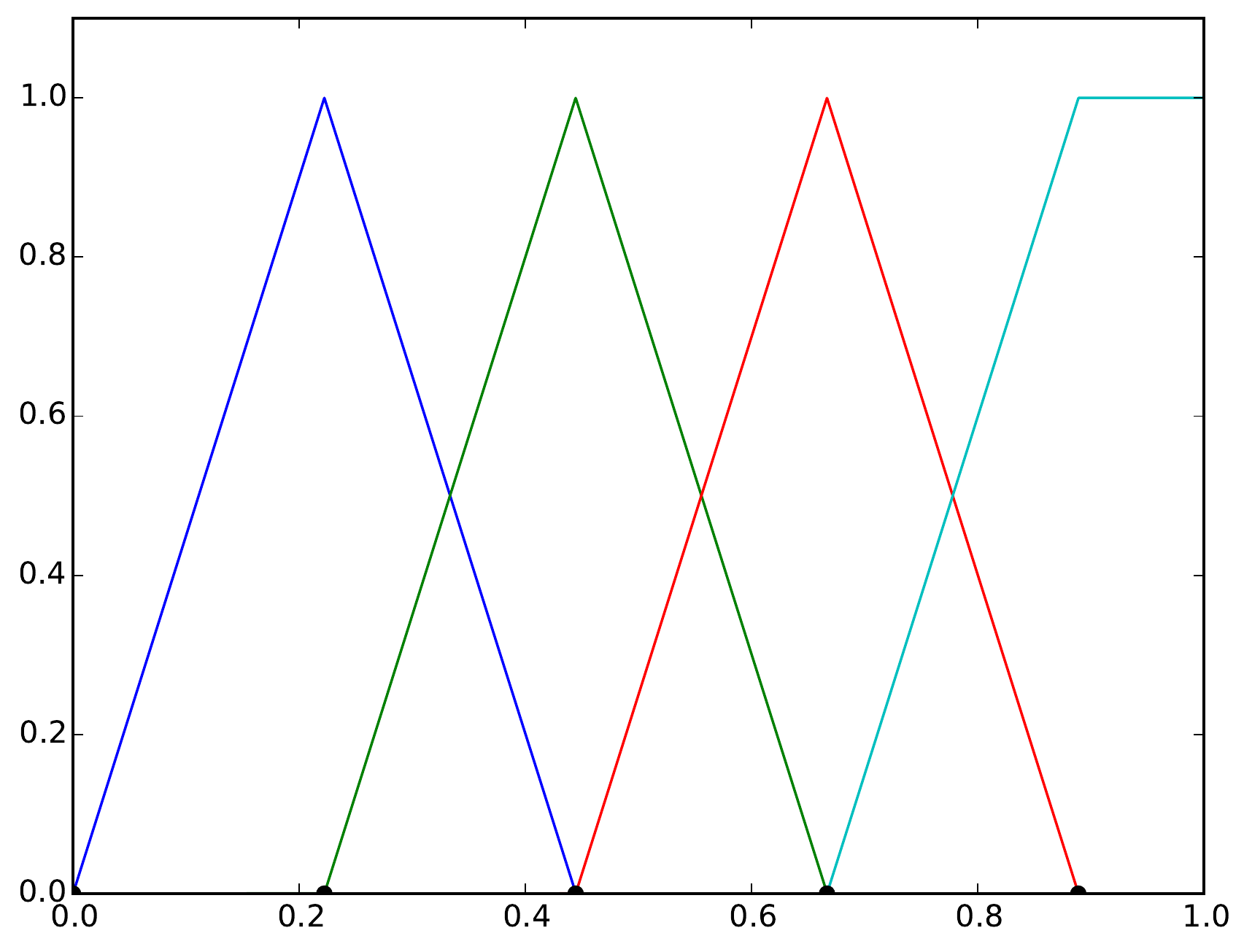}
      
      \includegraphics[width=.4\linewidth]{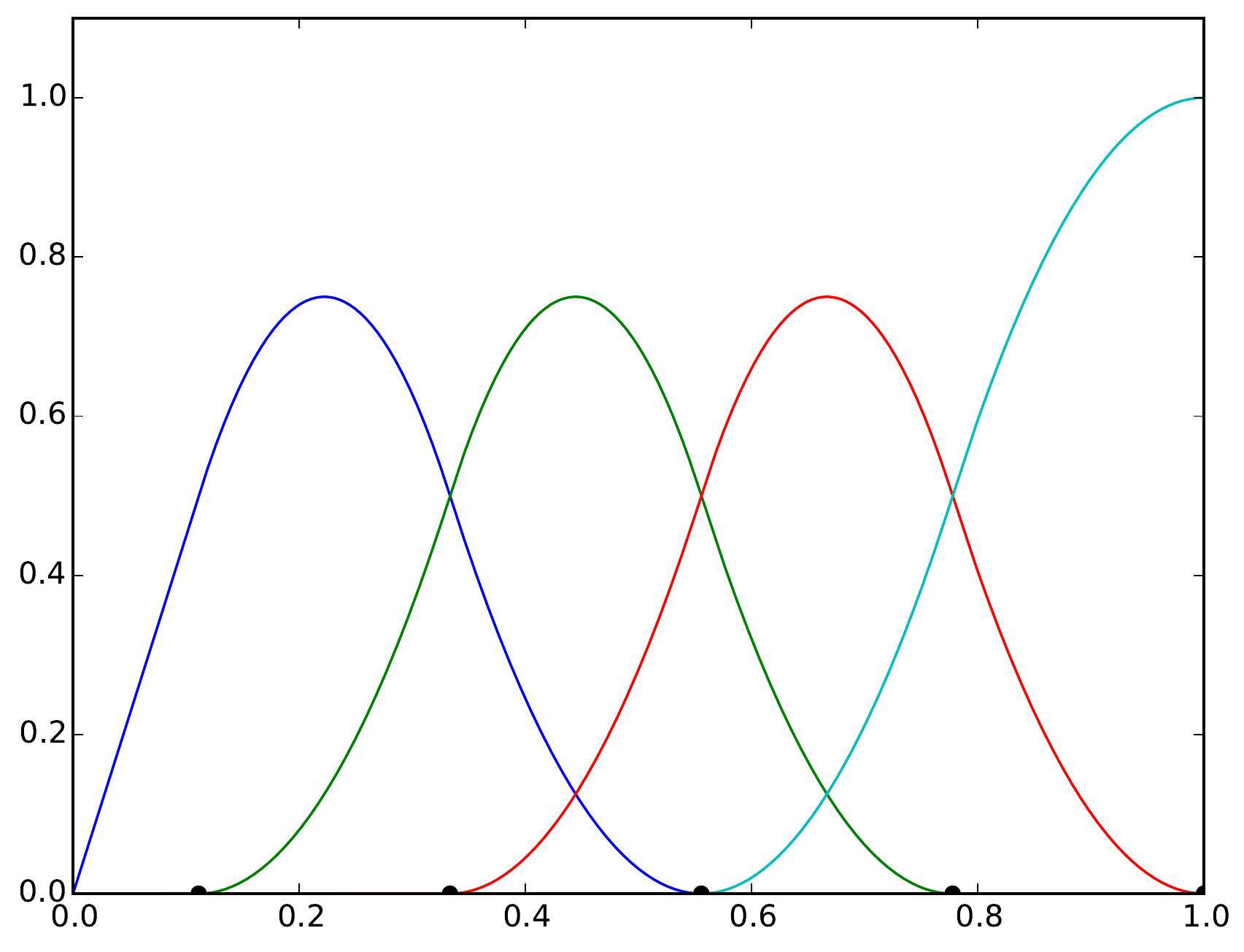}
      \includegraphics[width=.4\linewidth]{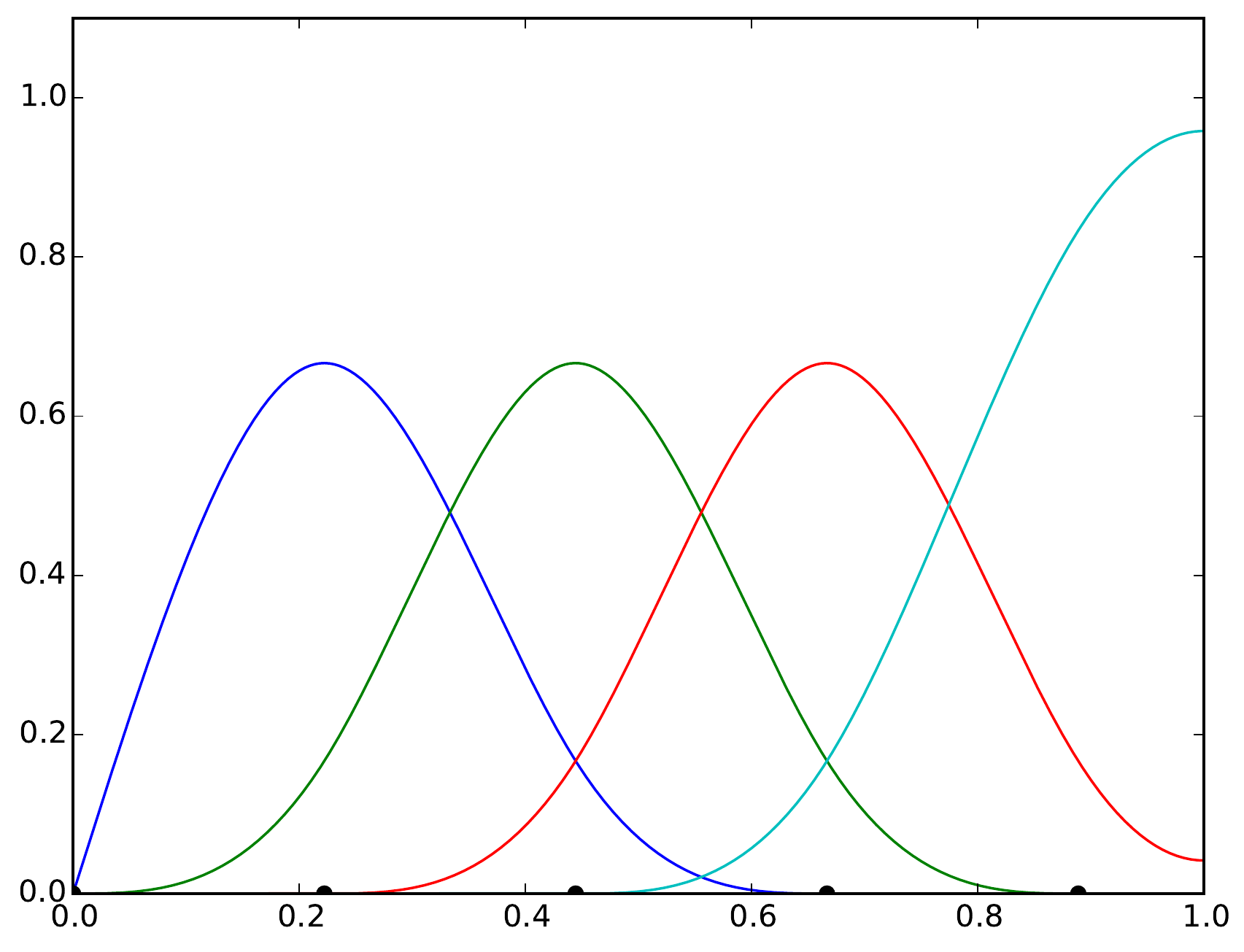}
  \caption{Basis functions for $S_{d,2}$.}
  \label{fig:Sd2}
\end{figure}

\section*{Acknowledgements}
We wish to thank the two referees for their careful reading of the
manuscript and their valuable comments which helped improve the paper. 
Espen Sande was supported by the
European Research Council under the European Union's Seventh Framework
Programme (FP7/2007-2013) / ERC grant agreement 339643.

\bibliography{nwidths}

\end{document}